\documentclass[11pt,a4paper]{amsart}
\usepackage[utf8]{inputenc}

\usepackage{amsfonts, amsmath, amssymb, amsthm, amscd}
\usepackage{amsmath,amsthm,amssymb,amsfonts}
\usepackage{hyperref}
\usepackage{mathtools}
\hypersetup{
  colorlinks   = true, 
  urlcolor     = blue, 
  linkcolor    = purple, 
  citecolor   = purple
}

\usepackage{xcolor}

\usepackage{cases}

\usepackage{anysize}
\marginsize{2.5cm}{2.5cm}{2.5cm}{2.5cm}
\usepackage[pdftex]{graphicx}
\usepackage{thmtools, thm-restate} 

\newtheorem{theorem}{Theorem}
\newtheorem{lemma}[theorem]{Lemma}

\newtheorem{corollary}[theorem]{Corollary}

\theoremstyle{remark}
\newtheorem{remark}[theorem]{Remark}
\newtheorem{question}[theorem]{Question}

\numberwithin{equation}{section}

\DeclareMathOperator{\Cr}{Cr}
\renewcommand{\epsilon}{\varepsilon}
\renewcommand{\phi}{\varphi}

\newcommand{\A}{\mathcal{A}}

\title{Crossing lemmas for $k$-systems of arcs}

\author{Alfredo Hubard}
\address{Universit\'e Gustave Eiffel, CNRS, LIGM, Marne-la-Vall\'ee, France}
\email{alfredo.hubard - at -univ-eiffel.fr}

\author{Hugo Parlier}
\address{University of Luxembourg, Department of Mathematics, Luxembourg}
\email{ hugo.parlier - at - uni.lu}

\thanks{The first named author was partially supported by the ANR grant ANR-19-CE40-0014 (project Min-Max).}
\thanks{The second named author was supported by the Luxembourg National Research Fund OPEN grant O21/16309996.}
\thanks{The first named author thanks Niloufar Fuladi and Arnaud de Mesmay for interesting discussions.}

\begin{document}

\begin{abstract} We show a generalization of the crossing lemma for multi-graphs drawn on orientable surfaces in which pairs of edges are assumed to be drawn by non-homotopic simple arcs which pairwise cross at most $k$ times. 
\end{abstract}

\maketitle

\section{Introduction}
 
The celebrated crossing lemma was discovered in the early 1980s by  Leighton, motivated by VLSI design and independently by Ajtai, Chv\'tal, Newborn, and Szemer\'edi who were answering questions raised by Turan, Erd\H{o}s and Guy. It provides a lower bound on the order of $m^3/n^2$ on the number of crossings of any planar drawing of a graph with $m$ edges and $n$ vertices. In the late 1990s, Szekely \cite{szekely} discovered a number of applications to incidence geometry problems for which a variant of the crossing lemma for multi-graphs is required. Szekely observed that the inequality for graphs implies an inequality for multi-graphs with a correcting factor of $\frac{1}{k}$, where $k$ is the largest number of edges between the same pair of vertices. In the last decade, crossing lemmas for multi-graphs, in which the correcting factor is not needed, have been given much attention. Starting with a question of Kaufmann and with the paper \cite{PaTo} of Pach and T\H{o}th, a series of papers have relaxed the necessary conditions for such a crossing lemma to hold. In this note, we observe that results from geometric topology, mainly due to Przytycki, can be used to obtain quantitatively improved, and more general, crossing lemmas for multi-graphs.

The link we use between graph drawings and topology is through systems of arcs. In this paper every arc is assumed to be simple. Given a surface with prescribed marked points (or punctures), we say that a family of arcs with endpoints on the marked points is a system of arcs, if no two arcs are homotopic to each other and homotopies are not allowed to pass through the marked points. In graph theoretical terms, the edges of a graph drawing are non-homotopic and simple if and only if they form an arc system on a sphere with marked points. If an arc system has the additional property that any two arcs pairwise intersect at most $k$ times, it will be called a $k$-system of arcs. The case of $k=1$ was studied in graph drawing under the name of single crossing non-homotopic multi-graph drawing (see \cite{Kaufmann}).

Our first result can be phrased in terms of the minimal number of intersection points of a $k$-system of arcs which, by analogy with the graph setting, we call its crossing number.

\begin{theorem}\label{crossing-karcs} If $\mathcal A$ is a $k$-system of $m$ arcs on an $n$-punctured sphere with $m>4n$, then \[\Cr(\mathcal A)\geq c_k\frac{m^{2+1/k}}{n^{1+1/k}},\] where $c_k \geq \frac{1}{10^6 k}$.
\end{theorem}

Upto the value of $c_k$ (which we made no effort to optimize) this theorem is best possible as explained below.  A slightly weaker version of this theorem was conjectured in \cite{Kaufmann}, and it was observed that it would follow from upper bounds on the size of arc systems.

For (orientable) surfaces of higher genus, we have the following more general result:

\begin{theorem}\label{genus-arcs} If $\mathcal A$ is a $k$-system of $m$ arcs on a surface of genus $g$ with $n$ punctures, denoted $S_{g,n}$, and $m > 16 n$ and $ n > 2^{17} g$, then:
\[\Cr(\mathcal A)\geq \frac{1}{10^8 k} \frac{m^{2+1/k}}{n^{1+1/k}}\text{      if      } m\leq \frac{5^{7k}k^k}{2^{5k}} \frac{n^{k+1}}{g^k}.\]
\[\Cr(\mathcal A)\geq  \frac{1}{2^{12}}\frac{m^2}{g} \text{      otherwise}.\]
\end{theorem}

We can observe the same dichotomy in terms of the number of edges with respect to the genus (see Theorem 4.1 in \cite{shahrokhi}). 

The proof of Theorem \ref{genus-arcs} combines Theorem \ref{crossing-karcs} with the following planarizing lemma of independent interest. Our proof of this lemma is similar to the proof of Theorem 6 in \cite{PaToSp}, but, even in the case of graphs, we improve the dependence in $g$. Define the set $\mathcal{A}(\alpha)$ to be the subset of arcs of $\mathcal A$ that intersect a family of disjoint simple closed curves $\alpha$, and denote by $d(v)$ the number of arcs incident to the puncture $v$.

\begin{lemma}\label{planarizing1} 
For any $k$-system of arcs $\mathcal{A}$ in $S_{g,n}$, there exists a family of disjoint simple closed curves $\alpha=\{\alpha_1,\alpha_2, \ldots \alpha_g\}$ such that, cutting along $\alpha$ yields a genus $0$ connected surface with $2g$ boundaries and  \[|\mathcal A(\alpha)| \leq 20\sqrt{ g(\Cr(\mathcal A)+\sum_v d(v)^2+n+2g)}.\] 

\end{lemma}
Notice that after pasting punctured disks along the boundaries of $S_{g,n}\setminus \alpha$, we obtain a surface homeomorphic to $S_{0,n+2g}$ in which $\mathcal{A}\setminus \mathcal A(\alpha)$ is a $k$-system. \\

We end the introduction with a question. Przytycki's work was motivated by problems about systems of closed curves in surfaces. Using Przytycki's theorem, Greene  \cite{greene1} showed that a $1$-system of simple closed curves in a closed surface of genus $g$ has at most $O(g^2 \log g)$ curves (see \cite{greene2} for $k>1$). The logarithmic factor is still mysterious, and conjecturally it is not necessary. 
 
\begin{question} Do there exist positive constants $c,c'>0$ such that any $1$-system of $m$ simple closed curves $\Gamma$ on a surface of genus $g$ with $m\geq c g$ satisfies $\Cr(\Gamma)\geq c' \frac{m^3}{g^2}$?
\end{question}

Notice that $\Cr(\Gamma)\leq {m \choose 2}$, so a positive answer to the above question would provide a quadratic upper bound on $1$-systems of curves in terms of the genus, showing the non-necessity of the logarithmic factor. We will revisit a variant of this question in a subsequent paper.

\section{Setup and preliminaries}

We consider a closed orientable finite type surface $S$ of genus $g\geq 0$ and with $n$ punctures (which we think of as marked points) and such that $\chi(S)=2-2g -n <0$. A simple arc is map of the unit interval into $X$ such that the map is an embedding from the open interval and such that the endpoints of the interval are mapped to the marked points. We consider arcs up to homotopy where homotopies are required to fix the marked points pointwise.

Intersection between homotopically distinct arcs $a,b$ is the integer $i(a,b)$ equal to the minimum number of transversal intersections between representatives of the homotopy classes of $a$ and $b$. If we restrict to simple representatives, this corresponds to 
$$
i(a,b) = \min \{ | a' \cap b'|  \mid a \sim a', b \sim b' \}
$$
where $\sim$ means homotopic and $| a' \cap b'|$ is the cardinality of the set outside the set of punctures. In particular, two arcs can share an endpoint (or both) and still have $0$ intersection. A family $\A$ of simple arcs on $S$ is called a {\bf $k$-system of arcs} if for every pair of arcs $a, a' \in \A$, $i(a,a') \leq k$.

\begin{remark}
Notice that there is no loss of generality by putting the arcs in minimal position in this definition and the hypothesis on homotopy classes of arcs is equivalent to the statement about arcs used in the introduction. 
\end{remark}

Systems of simple arcs on a surface are equivalent to drawings of multi-graphs (graphs in which we allow loops and multiple edges between vertices). The vertices of the graph are exactly the marked points of the surface. For instance, in our context, a simple graph drawn on the sphere corresponds to a system of arcs $\mathcal A$ on a sphere with $n$ punctures, for which any two punctures co-bound at most one arc, and there are no loops. More generally, a planar drawing is equivalent to a simple arc system on a sphere, but with a puncture without any arcs leaving from it corresponding to the point at infinity of the plane.

We will switch freely between the two terminologies. The set of punctures/vertices will often be denoted by $V$. The crossing number of a $k$-system of arcs $\mathcal A$, or of the corresponding multi-graph drawing $G$ (which we denote both by $\Cr(\mathcal A)$ and by $\Cr(G)$) is the number of intersection points between the interior of the arcs, again under the assumption that the interior of three or more arcs don't cross at a single point and that the intersection of two arcs is transversal.  In the following theorem (the first part is from \cite{PaTo}, the second part is from \cite{Suk}), we summarize results on crossing numbers under intersection conditions of multi-graphs that the case $k=1$ of Theorem \ref{crossing-karcs} improves and generalizes.


\begin{theorem}[Pach-Toth, Fox-Pach-Suk] \label{pach} Let $\mathcal A$ be a $1$-system of $m$ arcs without loops on a $(n+1)$-punctured sphere and such that any two arcs that share two endpoints bound a simple closed curve. Then the following hold:
\begin{itemize}

\item if any two arcs that share one endpoint have no other intersection, then $\Cr(\mathcal A )\geq  \frac{1}{10^{7}} \frac{m^3}{n^2}-4n$,

\item $\Cr(\mathcal A )\geq\frac{1}{10^{25}} \frac{m^3}{n^2 \log \frac {m} {n}}-8n$.

\end{itemize}
\end{theorem}

Theorem \ref{crossing-karcs} also generalizes one of the main results of \cite{wood}. There they assume that they are dealing with $k$-systems of arcs with the extra hypothesis that each of the arcs is the graph of a function from an interval of the real line into the reals. This condition is called \emph{$x$-monotone } in the combinatorics literature. They derive a crossing lemma similar to Theorem \ref{crossing-karcs} for $x$-monotone $k$-systems of arcs. 

The key new ingredient we use is the following theorem by Przytycki \cite{prz}.
 
\begin{theorem}[Przytycki]\label{arcs} If $\mathcal A$ is a $k$-system of arcs in a surface of Euler characteristic $\chi$, then
\begin{itemize} 
\item $|\mathcal A| \leq 2|\chi||\chi+1|$ for $k=1$, and this is best possible.
\item $|\mathcal A| \leq (k+2)!|\chi|^{k+1}$ for $k>1$, and the exponent $(k+1)$ is best possible.
\end{itemize}
\end{theorem}

Przytycki constructed examples of $k$-systems with more than $(\frac{|\chi|}{k+1})^{k+1}$ arcs. This was improved in \cite{wood} to $k$-systems of $x$-monotone curves in an $n$ punctured sphere with $\sum_{i=1}^{k+1} {n \choose i}$ arcs. Actually Przytycki construction is also in a $n$-punctured sphere and can be made with $x$-monotone curves. 

 As shown in Lemma 1 from \cite{Kaufmann}, the fact that the exponent of the second item is the best possible implies that, up to the value of $c_k$, Theorem \ref{crossing-karcs} is also the best possible for every $k\geq 1$. 

Combining Theorem \ref{arcs} with Turan's theorem, one obtains some simple corollaries about numbers of pairs that cross at most $k$ times. 

\begin{corollary}\label{turan-arcs} If $\mathcal A$ is a family of $m$ pairwise non-homotopic simple arcs on a surface of Euler characteristic $\chi$, then
\begin{itemize}
\item At least $\frac{m^2}{2}\frac{1}{3 |\chi|}-\frac{m}{2}$ pairs of arcs intersect at least once.
\item At least $\frac{m^2}{2} \frac{1}{2(|\chi|+1)|\chi|}-\frac{m}{2}$  intersect at least twice. 
\item At least $\frac{m^2}{2} \frac{1}{(k+2)! |\chi|^{k}}-\frac{m}{2}$ pairs of arcs intersect at least $k$ times.
\end{itemize}
\end{corollary}

\begin{proof}
Consider a graph $G_{k-1}(\mathcal A)$ that contains a vertex for each arc and such that two vertices are adjacent if the corresponding curves intersect in at most $(k-1)$ points. 
If this graph has more than $\frac{m^2}{2}(1-\frac{1}{(k+2)! |\chi|^k+1})$ edges, then by Turan's theorem, it contains a clique of size $\frac{1}{(k+2)! |\chi|^k}$, which contradicts Theorem \ref{arcs}. Passing to the complement, we obtain that the number of pairs of arcs intersecting at least $(k-1)$ times is greater than ${m \choose 2}-\frac{m^2}{2}(\frac{1}{(k+2)! |\chi|^k+1})$.
\end{proof}

Since $|\chi|=2g+n-2$, the implication of this corollary to the number of crossings is weaker than the previous theorems for most four-tuples of positive integers $m, n, k, g$.

\section{The planar case}\label{genus_0}

This section is essentially taken from \cite{Suk} and the idea goes back to \cite{PaToSp} and \cite{PaTo}. The main tool is the notion of \emph{branching bisection width} from \cite{PaTo} which we now restate in our language of arcs on punctured spheres.  Let $\mathcal A$ be  a $k$-system of arcs on the surface of genus $0$ with $n$ punctures $S_{0,n}$, which we abusing notation we also denote by $S_{0,V}$, where $V$ is the set of punctures $V$ (so $|V|=n$). For a subset $V' \subset V$, we denote by $S_{0,V'}$ the surface $S_{0,V}$ after we fill the punctures in $V\setminus V'$. In graph theoretical terms, we forget about the vertices that are not in $V'$. 

Define $bb(\mathcal A)$ to be the least number of arcs that need to be removed so that we obtain two sets of punctures $V_1,V_2$, each with at least $\frac{n}{5}$ of them, and such that no remaining arc is incident to a puncture in $V_1$ and a puncture in $V_2$. The remaining arcs are partitioned into two $k$-systems $\mathcal A_1$ in $S_{0,V_1}$ and $\mathcal A_2$ in $S_{0,V_2}$. In graph theoretical terms, after erasing a set of edges, we have two disjoint multi-graphs with at least $n/5$ vertices each. We now need the following estimate on the size of $bb(\mathcal A)$ from \cite{PaTo}:

\begin{theorem}\label{bisection} For any $k$-system $\mathcal A$ in $S_{0,V}$ 
\[bb(\mathcal A)\leq 20 \sqrt{(\Cr(\mathcal A) + \sum_{v \in V} d(v)^2+ |V|}) \]
\end{theorem}

The idea of the proof is to replace each puncture-vertex $v$ by a grid of side length $d(v)$,  and add a puncture at each crossing between two arcs, obtaining a planar graph to then apply a version \cite{alon} of the famous separator theorem to this planar graph. The construction of this embedded graph of maximum degree $4$ will be used in the proof of lemma \ref{planarizing}. 

The original theorem was stated for the conditions of Theorem \ref{pach}, but the proof easily generalizes to $k$-systems. The only detail that needs to be adapted is at the end of the proof. We refer to  \cite{PaTo} for more details, and just sketch the main idea here. In the last step of the proof, we are given a simple closed curve $\gamma$ on the sphere that separates the sphere into two disks $D_1$ and $D_2$. This curve is transverse to every arc in $\mathcal A$, and does not pass through any puncture. We put $V_i:=D_i\cap V$ and $\mathcal A_i$ for the set of arcs connecting two punctures in $V_i$ (for $i=1,2$).  The problem that arises is that a pair of arcs in $\mathcal A_1$ might be homotopy equivalent in $S_{0,V_1}$ (but not in $S_{0,V}$). The important observation here is that any such two pairs of homotopic arcs in $S_{0,V_1}$  must have distinct endpoints and must intersect at least four times \footnote{In the original theorem it was assumed that any pair of arcs that share two common endpoints bounds a simple curve, but this is not necessary for the statement to be true.}. From this observation, it follows that the number of pairs of curves in $\mathcal A_1$ that are homotopically equivalent in $S_{0,V_1}$ is bounded above by $\sqrt{\Cr(\mathcal A)}$ and similarly for the pairs in $\mathcal A_2$ that are homotopically equivalent in $S_{0,V_2}$. All in all, the arcs that we erase to obtain the branching bisection of $\mathcal A$ are those intersected by $\gamma$, plus one arc in $\mathcal A_1$ for each pair of arcs that became homotopic in $S_{0,V_1}$ (and similarly for $\mathcal A_1$), and we account for these arcs in our estimate of $bb(\mathcal A)$ by increasing the multiplicative constant in the right hand side of the inequality. 

\begin{proof}[Proof of Theorem \ref{crossing-karcs}]

Similar proofs can be found in  \cite{PaToSp,PaTo}.  A general framework that axiomatizes this proof can be found in \cite{Kaufmann}, for concreteness we adapt the presentation of \cite{Suk} to the case of general $k$.\\

\noindent{\it Vertex-splitting.} Define $\Delta=\lceil \frac{2m}{n} \rceil$, and split each vertex that has degree larger than $\Delta$ into vertices of degree $\Delta$ and possibly one vertex of smaller degree. Let $G'$ be the new topological multi-graph with identical crossings and $n'$ vertices, where $n\leq n' \leq 2n$. 

Now we set up an inductive procedure. At each step, $\mathcal F_i$ is a family of subgraphs of $G'$, and we apply Theorem \ref{bisection} to some of them. Set 
\[t=\frac{10^{-4}}{k+2}\frac{m^{1+\frac{1}{k}}}{n^{1+\frac{1}{k}}} \mbox{
and } \mathcal F_0=\{G\}.\] For each $i$,  $\mathcal F_i$ is a vertex disjoint family of induced multi-graphs, each $H \in \mathcal F_i$ is drawn on $S_{0,V(H)}$ so that its edges are a $k$-system of arcs, and each satisfies at least one of the following:

\begin{itemize}
\item $\Cr(H) \geq t \,e(H)$ (many crossings),
\item $v(H) \leq (\frac{4}{5})^i n'$ (the graph is small).
\end{itemize}

For $i=0$, the graph has $n'$ vertices. At step $i$, for a graph $H \in \mathcal F_i$, if $\Cr(H)\geq t \, e(H)$ or has less than $(\frac{4}{5})^{i+1} n'$ vertices, then we move it to $\mathcal F_{i+1}$. Otherwise, the number of vertices of $H$ is between $(\frac{4}{5})^{i} n'$, and $(\frac{4}{5})^{i+1} n'$ and we apply Theorem \ref{bisection} to $H$, erasing at most 
\[20 \sqrt{t e(H) +\Delta e(H) + v(H)}\leq 40  (\sqrt{t e(H)} + \sqrt{v(H)})\] edges to obtain topological multi-graphs $H_1\to S_{0,V(H_1)}, H_2\to S_{0,V(H_2)}$, with $E(H_1,H_2)=\emptyset$ and such that the edge set of each is a $k$-system.

Each $H \in \mathcal F_{i}$ that is not moved to $\mathcal F_{i+1}$ has at least $\left(\frac{4}{5}\right)^{i+1} n' $ vertices, so there are at most $\left(\frac {5}{4}\right)^{i+1}$ such graphs. Since $\sum_{H \in \mathcal F_i} e(H)\leq m$ and $\sum_{H \in \mathcal F_i} v(H)\leq n'$, we have \[\sum_{H \in \mathcal F_i} \sqrt{e(H)} \leq \sqrt{m} \left(\frac{5}{4}\right)^{\frac{i+1}{2}}\mbox{ and }\sum_{H \in \mathcal F_i} \sqrt{v(H)} \leq \sqrt{n'} \left(\frac{5}{4}\right)^{\frac{i+1}{2}}.\] In step $i$, we erased a total of at most \[40 \left(\frac{5}{4}\right)^{\frac{i+1}{2}} (\sqrt{t m}+\sqrt{n'})\]
edges. If we stop at the last $j$ such that $\left(\frac{5}{4}\right)^{j/2}\leq 10^{-3} \sqrt{\frac{m}{t}}$, then at most
\[\sum_{i=0}^j 40 \left(\frac{5}{4}\right)^{\frac{i+1}{2}} [\sqrt{t m}+\sqrt{n'}] \leq 500 \left(\frac{5}{4}\right)^{j/2} \sqrt{tm}\leq \frac{m}{2}\] are erased in the inductive process. Each graph $H \in \cup_{i=1}^j \mathcal F_j$ satisfies $\Cr(H)\geq t\, e(H)$, or has at most $\left(\frac{5}{4}\right)^j n' \leq 10^7 t \frac{n'}{m}$ vertices which, by Theorem \ref{arcs}, translates to at most 
\[10^7(k+2)! [t \frac{n'}{m}]^{k+1}\leq m \left(\frac{(k+2)^2}{10 e^{k+1}}\right) \frac{m^\frac{1}{k}}{{n'}^{1+\frac{1}{k}}}\leq
m \left(\frac{1}{10} \frac{(k+2)^2}{e^{k+1}} \frac{(k+2)^{1+\frac{2}{k}}}{e^{1+\frac{2}{k}}}\right)\leq\]\[ \leq m \left(\frac{1}{10}\frac{(k+2)^{3+\frac{2}{k}}}{e^{k+2+\frac{1}{k}}}\right)\leq \frac{m}{4}\] of them are in the very small graphs. 

The erased edges in the inductive process are at most $\frac{m}{2}$ by Theorem \ref{bisection} and the edges in very small graphs at most $\frac{m}{4}$ by Theorem \ref{arcs}. The remaining $\frac{m}{4}$ edges are in multi-graphs such that $\Cr(H)\geq t\, e(H)$, hence we can conclude that $\Cr(G)\geq \frac{t}{4} m\geq \frac{1}{10^{5} (k+2)}  \frac{m^{2+1/k}}{n^{1+1/k}}.$
\end{proof}

This concludes the proof in the planar case. We now proceed to handling surfaces of genus $g>0$.

\section{The general case}\label{proofs}

We need one extra step to go from genus $g$ to genus $0$. For this we use the following result \cite{djidjev}. 

\begin{theorem}[Djidjev-Venkatesan]\label{planarizing} For any (simple) connected graph $G$ with $m$ edges of degree at most $d$, embedded on an orientable surface of genus $g$, there exists a set of at most $4\sqrt{2dgm}$ edges whose removal makes $G$ planar.
\end{theorem}

Let $\mathcal A$ be a systems of arcs, and $\alpha$ a multi-curve (a disjoint family of simple curves), we denote by $A(\alpha)$ the subsystem of arcs that intersects $\alpha$. Notice that in the previous theorem, if we dualize, and consider the dual of the edges that we remove, we obtain a multi-curve.  We now prove Lemma \ref{planarizing1}, stated in the introduction, that allows us to planarize drawings.

\begin{proof}[Lemma \ref{planarizing1}]
As before, we assume that $|V|=n$,  we think of $\mathcal A$ as a drawing of a multigraph $G$ and for each of vertex in $G$ put $\bar{d}(v)$ to be the degree of $v$ in $G$ if $v$ is not isolated, and $1$ if $v$ is isolated. 
Blow up each vertex $v_i$ of $G$ to a square grid of side length $\bar{d}(v_i)$, we denote the vertices of this grid by $V_i$, and we denote the corresponding blown-up graph by $G'$. We will construct the edges of $G'$ so that no two edges join the same pair of vertices (so that $G'$ is a graph and not a multi-graph). The blown-up graph $G'$ has $\sum_v \bar{d}(v)^2$ vertices.  The edges of $G'$ are of two types, firstly for each edge in $G$, there is an edge in $G'$, secondly,  for each $i$, $G'$ contains the edges of the aforementioned grid within the vertex set $V_i$. By making the drawings of the grids small enough and a judicious choice of arcs, $G'$ can be drawn so that $\Cr(G)=\Cr(G')$. More precisely, if $v_i$ is not isolated, we choose a side on each square-grid, and draw the $\bar{d}(v_i)$ edges incident $V_i$ to be incident to this side of the grid $V_i$. Each of the edges of $G'$ that corresponds to an edge of $G$ is incident to a different vertex of $V_i$. Notice that the maximum degree of $G'$ is $4$. It is easy to see that by constructing the embedded grid on $V_i$ in a small neighborhood of $v_i$, one can route the edges of $G'$ so that there exist a continuous map  $\phi_1 \colon S_{g,\sum_v \bar{d}(v)^2} \to S_{g,n}$, which restricts to a bijection between the crossings of $G$ and those of $G'$, the inverse image of each edge in $G$ is an edge in $G'$ which is not contained in a square grid, and for each $i$, $\phi_1$ maps $V_i$ to $v_i$. 

If the edges of $G$ are a $k$-system then the edges of $G'$ are a $k$-system. Now we are going to introduce a vertex at each crossing of the drawn graph $G'$ to obtain a new graph $H$. This is a graph (not a multi-graph) and it is embedded on a surface of genus $g$. It has $N=\Cr(G)+\sum_v  \bar{d}(v)^2$ vertices and its maximal degree is $4$ (we assume no three of the original arcs go through the same point). In other words the edges of $H$ form a $0$-system of arcs in the surface $S_{g,N}$, and there is a map $\phi_2 \colon S_{g,N} \to S_{g,\sum_v \bar{d}(v)^2}$, that sends the union of the edges of $H$ to the union of edges of $G'$.

Notice that, the grids in $G'$ remain being grids in $H$ because they were already embedded. By Euler's formula $H$ has at most $3N+2g$ edges which we denote by $E$. By Theorem \ref{planarizing}, there exists a multi-curve $\alpha$ intersecting a set $E(\alpha)$ of at most $20\sqrt{g(\Cr(G)+\sum_v \bar{d}(v)^2+2g})$ edges of $H$, so that after erasing these edges we obtain a planar graph. We are now going to modify $\alpha$ to avoid intersecting the grid edges. For each $i$, let $s_i$ be a simply connected open disk of the surface that contains the grid $V_i$. The restriction of $\alpha$ to $s_i$ is a union of arcs. Indeed any closed curve completely contained in $s_i$ is contractible, and we might erase it. For each subarc of $\alpha$ that enters $s_i$, consider another arc that starts and ends near the same points of $\partial s_i$ and stays on the boundary of $s_i$.  It is not hard to see that this surgery of $\alpha$ can be done without increasing the number of edges of $H$ crossed by by $\alpha$ (see a similar analysis in \cite{PaTo}). 
After modifying $\alpha$ repeatedly, $\alpha$ avoids all the grid edges. In other words, each of edge of $E(\alpha)$ corresponds to a subarc of an edge in $G$ through $\phi_1 \circ \phi_2$, hence $|\mathcal A(\alpha)|=|E(\alpha)|$. By Theorem \ref{planarizing} the surface with boundary $S_{g,n} \setminus \alpha$ has genus $0$. If we replace the boundary components with disks with one puncture each, we trace $\mathcal A \setminus \mathcal A(\alpha)$ we obtain a $k$-system of arcs on a sphere with $n+2g$ punctures in total. Indeed, if two arcs $a,a' \in \mathcal A \setminus \mathcal A(\alpha)$ are homotopic in $S_{0,n+2g}$ then their $\phi_1\circ \phi_2$-images are homotopic in $S_{g,n}$.
\end{proof}

With this in hand, we can pass to the proof of the general case.

\begin{proof}[Theorem \ref{genus-arcs}]
We can assume $g>0$ as otherwise the result is Theorem \ref{crossing-karcs}. We begin, like in the proof of Theorem \ref{crossing-karcs}, by splitting vertices. Let $G'$ be a new topological multi-graph with identical number of edges and crossings, with $n'$ vertices with $n\leq n' \leq 2n$ each of which has degree at most $\Delta:=\lceil \frac{2m}{n} \rceil$. We abuse notation and still call $\mathcal A$ the system of arcs.

By Lemma \ref{planarizing1}, we obtain a subsystem of arcs $\mathcal A (\alpha)$ such that, if we put $e:=|\mathcal A (\alpha)|$, then

\[e\leq 20\sqrt{g (\Cr(\mathcal A)+\sum_v d(v)^2+n'+2g-2)}\leq 20\sqrt{g (\Cr(\mathcal A)+2n' (\Delta^2+1)+2g)},\]
which implies
\[Cr(\mathcal A)\geq \frac{e^2}{2^{9} g}-4\frac{m^2}{n}-2n-2g.\]

The system of arcs $\mathcal A \setminus \mathcal A(\alpha)$ can be thought of as a $k$-system of arcs in a $2n+2g$ punctured sphere. Provided that $m-e>8n$, we can apply Theorem \ref{crossing-karcs} which, since we assumed  $ n > 2^{17} g$, implies  \[\Cr(\mathcal A \setminus \mathcal A(\alpha)) \geq \frac{1}{10^6k}\frac{(m-e)^{2+1/k}}{(2n+2g)^{1+1/k}} \geq \frac{1}{10^8k}\frac{(m-e)^{2+1/k}}{(n)^{1+1/k}}  .\]

If $e>m/2$ we rely on the first inequality and if $e\leq m/2$ we rely on the second one, and since $m > 16 n$ and $ n > 2^{17} g$, we have
\[\Cr(\mathcal A)\geq \min \left( \frac{m^2}{2^{12} g},\frac{1}{10^8 k}\frac{m^{2+1/k}}{n^{1+1/k}}\right ).\]

Now if $m\geq \frac{5^{7k}}{2^{5k}} \frac{k^k}{g^k} n^{k+1}$, the first term is the smallest one, and otherwise we can use the second term as a lower bound. 
\end{proof}

\end{document}